\documentclass{aims} 
\usepackage{amsmath}
\usepackage{paralist}
\usepackage[misc]{ifsym}
\usepackage{epsfig} 
\usepackage{epstopdf} 
\usepackage[colorlinks=true]{hyperref}
\hypersetup{urlcolor=blue, citecolor=red}
\usepackage{graphicx}
\usepackage{subcaption}
\usepackage{caption}
\usepackage{mathrsfs}
\usepackage{enumerate}
\usepackage{geometry}
\allowdisplaybreaks

\def\currentvolume{X}
 \def\currentissue{X}
  \def\currentyear{200X}
   \def\currentmonth{XX}
    \def\ppages{X-XX}
     \def\DOI{10.3934/xx.xxxxxxx}

\newtheorem{theorem}{Theorem}[section]

\newtheorem{lemma}[theorem]{Lemma}
\newtheorem{proposition}[theorem]{Proposition}

\theoremstyle{definition}

\title[Dynamical System for Neural Network Activity in Drosophila]
{A dynamical system approach to modeling neural network activity in Drosophila orientation} 

\author[S. Ismail, B. Ambrosio, M.A. Aziz-Alaoui and Y. Souleiman]{}

\subjclass{Primary: Primary 92B20; Secondary 34F05}
\keywords{Dynamical Systems, Complex Systems, Networks, Neuroscience, Drosophila Orientation}

\thanks{The corresponding author gratefully acknowledges L. Abbot for inspiring discussions that motivated this work.}

\thanks{$^*$Corresponding author: Benjamin Ambrosio}

\begin{document}
\maketitle

\centerline{\scshape
S. Ismail$^{{\href{mailto:said.ismail-abdi@etu.univ-lehavre.fr }{\textrm{\Letter}}}1,2}$, B. Ambrosio$^{{\href{mailto:benjamin.ambrosio@univ-lehavre.fr}{\textrm{\Letter}}}*1,3}$,
M.A. Aziz-Alaoui$^{{\href{mailto:aziz.alaoui@univ-lehavre.fr }{\textrm{\Letter}}}1}$ and Y. Souleiman$^{{\href{mailto:yahyeh_souleiman_isman@univ.edu.dj  }{\textrm{\Letter}}}2}$
}

\medskip

{\footnotesize
 \centerline{$^1$Normandie Univ., LMAH UR 3821, 25 rue Philippe Lebon, Le Havre, 76600, Normandie, France}
} 

\medskip

{\footnotesize
 \centerline{$^2$Laboratory of Analysis, Modeling, and Optimization (LAMO), CRMN Center, University of Djibouti, Campus Balbala, Djibouti}
}
{\footnotesize
 \centerline{$^3$The Hudson School of Mathematics, 244 Fifth Avenue, New York, 10001, New York, USA}
} 
\bigskip

 \centerline{(Communicated by Handling Editor)}


\begin{abstract}
We introduce and analyze a class of neural network models motivated by the Drosophila central complex nervous system, designed to capture the emergence and dynamics of orientation-selective activity bumps. Starting from a biologically inspired ring-connectivity model, we derive a simplified reduced model of recurrent neural activity that supports stable, localized patterns encoding angular position during the fly’s flight orientation.
We first study the deterministic dynamics and identify parameter regimes ensuring existence and global stability of bump solutions. We then extend the framework to a stochastic setting, incorporating both additive Brownian noise and a Markovian switching mechanism representing time-varying external cues. The resulting system is a switching diffusion with piecewise linear drift, for which we establish well-posedness, characterize the infinitesimal generator, and prove the existence of an invariant measure. Numerical simulations in low and high dimensions illustrate the robustness of the bump attractor under noise and switching stimuli, as well as the convergence toward the predicted stationary states. These results provide a mathematically tractable framework for understanding how population activity in the insect central complex encodes heading direction, with the location of the population-activity pattern representing the animal's current heading in response to external stimuli, while stochastic variability influences the precision and stability of this directional representation.
\end{abstract}


\section{Introduction}

This article presents a mathematical description and analysis of neuronal activity in the Drosophila ring network. Over the past two decades, experimental studies have established the central role of the Drosophila central complex in orientation and navigation during both walking and flight. The central complex comprises the ellipsoid body, the protocerebral bridge, the fan-shaped body, and the noduli.

Individual EPG (E for Ellipsoid Body, P for Protocerebral Bridge, and G for Gall or Glomerulus) neurons exhibit mixed input–output (dendritic) terminals within a single wedge of the ellipsoid body and axonal terminals in a single glomerulus of the protocerebral bridge. During locomotion, the population activity of EPG neurons forms a localized calcium bump of neuronal activity in the ellipsoid body, with corresponding representations in the left and right protocerebral bridge. These activity patterns shift coherently across structures, encoding the fly’s angular heading relative to external reference cues (Fig. \ref{fig:centralcomplex}). Key advances were as follows.

The ellipsoid body was first identified as a neural compass encoding heading direction through a ring-shaped activity pattern characterized by a stable, localized bump whose dynamics integrate angular velocity and align with visual landmarks \cite{Seelig2015}. Green \textit{et. al.} \cite{Green2017}, subsequently elucidated the circuit architecture underlying angular path integration in this system, including the protocerebral bridge, thereby providing experimental validation of theoretical models of heading representation. More recently, Cheng \textit{et. al.} (\cite{Cheng2021}) demonstrated that the central complex performs vector arithmetic by representing heading and movement as sinusoidal population activity in the fan-shaped body, enabling accurate spatial updating when traveling and heading directions diverge, such as during sideways walking.
In the present article, we focus on the mathematical properties of a model that provides an efficient and tractable description of calcium-activity bumps in central complex neuronal populations and which can be directly associated with the two-dimensional orientation of the fly. We derive and subsequently provide a rigorous analysis of a simplified version of the ring model described in \cite{Abbot2021,Cheng2021,kim2017}, originally introduced in \cite{Ben1995,GOLD2004} in the context of neuronal activity in the visual cortex.

\begin{figure}
    \centering
    \includegraphics[width=0.5\linewidth]{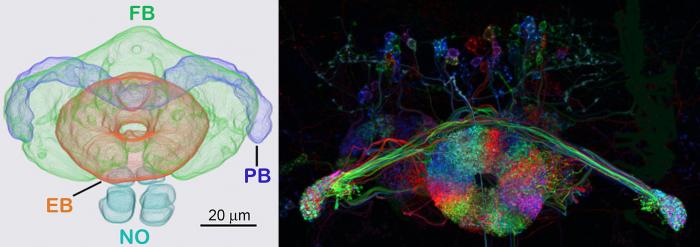}
    \includegraphics[width=3cm]{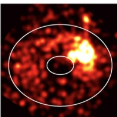}
    \caption{Structural, neuronal and functional  organization of the central complex. (Left) Schematic reconstruction of major neuropils, including the forebrain (FB), protocerebral bridge (PB), ellipsoid body (EB), and noduli (NO). Scale bar: 20 µm. (Center) Confocal image of multicolor-labeled neurons showing projection patterns and connectivity across central complex regions. (Right) Snapshots of compass-neuron population. Credit: Janelia Research Campus (HHMI); central complex anatomy adapted after  \cite{Kim2019Heading} and \cite{Wolff2018CentralComplex}, CC BY 4.0. }
    \label{fig:centralcomplex}
\end{figure}

We consider the following equation:

\begin{equation}
\label{eq:1}
\tau \dfrac{dr_i}{dt}=-r_i+F\left(\sum_{j=0}^{N-1}A_{ij}r_j+V_i\right)=-r_i+F(x_i),
\end{equation}

with $$ \ A_{ij}= J_0+J_i \cos(\theta_i-\theta_j)+L\cos\left(\theta_i-\theta_j+\frac{\pi}{4}\right)+R\cos\left(\theta_i-\theta_j-\frac{\pi}{4}\right), \ i\in \{0,...,N-1\}.$$

The different terms of the equation can be characterized as follows:
\begin{itemize}
\item[-]  $r_i=r_i(t)$ is the calcium concentration at angle $\theta_i$, with $\theta_i=i\dfrac{2\pi}{N}$.
\item[-]   $\tau$ is the leakage time scale.
\item[-]  $A$ is a square matrix representing synaptic interaction strength, with $A_{ij}$ representing the strength of the connection from angle $\theta_j$ (or $\theta_j=j2\pi/N$) to angle $\theta_i$.
\item[-]  $V_i=V_i(t)$ is an external input at angle $\theta_i$, which may or may not vary over time.
\item[-]   $x_i=x_i(t)= \sum_{j=0}^{N-1} A_{ij}r_j(t)+V_i(t)$ is the total input at angle $\theta_i$ over time.
\item[-]   $F: \mathbb{R}\rightarrow\mathbb{R}$ is a piecewise continuous linear function defined by:
\end{itemize}

\begin{equation}
F(x) =
\begin{cases}
x & \text{if } x\geq 0, \\
0 & \text{otherwise}.
\end{cases}
\label{eq:A_def}
\end{equation}

In the context of computational neuroscience, it is useful to view model \eqref{eq:1} as a coarse-grained firing-rate representation inspired by the leaky integrate-and-fire (LIF) model \cite{gerstner2014spiking,dayanabbott2001,izhikevich2007,lapicque1907}. We briefly recall here that the leaky integrate-and-fire (LIF) model is a foundational simplification in which a neuron’s membrane potential integrates incoming synaptic input while simultaneously decaying (leaking) back toward a resting state, producing a spike whenever a fixed threshold is crossed. Compared to the benchmark Hodgkin-Huxley model \cite{hh1952}, this abstraction captures key features of neuronal spiking dynamics while remaining mathematically and computationally tractable, making it widely used in large-scale neural simulations, such as for example in models of the visual cortex v1 \cite{csy2016,csy2018}. Here, in \eqref{eq:1}, the neuron receives excitatory drive from the network when the sum inside 
$F$ is positive, while simultaneously undergoing passive decay toward its resting state. In this representation, model \eqref{eq:1} should be regarded as a coarse-grained, rate-based analogue of LIF dynamics rather than as an explicit membrane-potential model. The membrane voltage is not included as a separate state variable; instead, its integration of synaptic inputs and the resulting spike generation are represented implicitly by the total drive 
\[ x_i(t)=\sum_{j=0}^{N-1}A_{ij}r_j(t)+V_i(t) \]
and the threshold-linear transfer function $(F)$. When $x_i>0$, the network input produces neuronal activation and hence an increase in the calcium-related variable $r_i$, which serves as a proxy for the underlying electrical and spiking activity. At the same time, the term $-r_i$ describes the passive relaxation of this activity on the timescale $\tau$. Thus, although $r_i$ is not the membrane voltage, its dynamics reflect the balance between input-driven neuronal activation and leakage that characterizes the LIF framework. Another choice, not considered here would be a $tanh$ type function for $F$. In order to proceed with the mathematical analysis, we seek to simplify equation \eqref{eq:1} while preserving its ability to produce a bump and to represent neuronal activity during the fly’s flight. Therefore, we assume that
$J_0=L=R=0,$
$\tau=1,\ 0<J_i=\delta>0,$
so the connectivity matrix $A$ is therefore reduced to $A=A_{ij}=\delta \cos(\theta_i-\theta_j).$
The network model then becomes:
\begin{equation}
\label{eq:linReLu}
\dot{r}(t) = -r(t) + F\!\big( A r(t) + V \big),
\end{equation}

where
\[
A_{ij} = \delta \cos(\theta_i - \theta_j), \qquad i,j \in \{0,\ldots,N-1\},
\]
and
\[
r(t)=\begin{pmatrix}
r_{0}(t)\\
r_{1}(t)\\
\vdots\\
r_{N-1}(t)
\end{pmatrix},
\qquad
V=\begin{pmatrix}
V_{0}\\
V_{1}\\
\vdots\\
V_{N-1}
\end{pmatrix}.
\]

Equation \eqref{eq:linReLu} is the main equation studied in this article. It describes a network of ordinary differential equations (ODEs), where the number of units $N$ typically ranges from a few dozen \cite{Ishida2026}.

We now introduce a natural stochastic extension of this model. Since it is reasonable to assume that the fly follows a cue while experiencing fluctuations in its motion, we add a Brownian noise term to each component $r_i$. We further assume that $V_i=0$ everywhere except at a single location representing the cue. Finally, we let the cue evolve according to a continuous-time Markov chain $I_t$.

The resulting stochastic differential equation is given by
\begin{equation}
\label{eq:sdeN}
dr_{i}(t)
=
\Bigg(
-r_{i}(t)
+
F\Big(
\sum_{j=0}^{N-1} A_{ij} r_{j}(t)
+
V_i(I_t)
\Big)
\Bigg)\,dt
+
\sigma_i\, dB_i(t),
\end{equation}
where $I_t$ is a continuous-time Markov chain on $\{1,\ldots,N\}$ with rate $\lambda>0$. The input vector $V(I_t)$ is defined such that $V_j = c > 0$ if $j = I_t$, and $V_j = 0$ otherwise.
The remainder of the article is organized as follows. In Section 2, we present the mathematical proofs. We first consider the case $N=2$, and then extend the analysis to the general case. Section 3 is devoted to numerical simulations.
\section{Qualitative Analysis}

In this section, we first study the case $N=2$ and then treat the general case.

\subsection{Qualitative Analysis for the case $N=2$}

\subsubsection{Deterministic equation ($N=2$)}

We begin our mathematical analysis with the case $N=2$, for which a more detailed treatment is possible. This simplified setting already captures key features of the system and helps motivate the choice of parameters. Moreover, we restrict our attention to biologically relevant regimes by assuming $V_1>0$, $V_2=0$. Under these assumptions, the deterministic system takes the form of
\begin{equation}
\label{eq:N=2}
\left\{
\begin{array}{rcl}
r_1' &=& -r_1 + F\big(\delta(r_1 - r_2) + V_1\big), \\
r_2' &=& -r_2 + F\big(\delta(r_2 - r_1)\big),
\end{array}
\right.
\qquad
F(x) = \max(x,0),
\end{equation}
where $\delta>0$, $V_1>0$.
Thus, system \eqref{eq:N=2} is the (N=2) reduction of model \eqref{eq:1}. The two units represent two neurons or neuronal populations tuned to opposite angular directions. The recurrent inputs therefore take the forms $\delta(r_1-r_2)$ and $\delta(r_2-r_1)$. Here, $\delta>0$ is the coupling strength and describes a competition between the two populations: the activity of each population reinforces its own response and suppresses the response associated with the opposite direction. Moreover, $V_1>0$ represents an external directional cue acting on the first population, whereas $V_2=0$ means that no direct cue is applied to the second population. The rectification $F(x)=\max(x,0)$ ensures that only positive total inputs contribute to neuronal activation, while the terms $-r_i$ describe passive decay of the calcium-related activity. This system models a biological situation in which the fly is exposed to an external stimulus biased toward a preferred direction. This system models a biological situation in which the fly is exposed to an external stimulus biased toward a preferred direction. As stated in the theorem below, the parameter regime $0<\delta<1$ ensures that the coupled system has a unique globally asymptotically stable equilibrium in the positive quadrant, which means that the the population receiving the external cue is selected without the appearance of additional competing equilibria. We therefore focus on $0<\delta<1$ as the biologically relevant weak-coupling regime. The following theorem holds.

\begin{theorem}
Assume that \(V_1>0\) and \(0<\delta<1\). Then equation
\eqref{eq:N=2} admits a unique equilibrium in the non-negative
quadrant, namely
\[
E_W
=
\left(
\frac{V_1}{1-\delta},
\,0
\right).
\]
Moreover, \(E_W\) is globally asymptotically stable in the
non-negative quadrant.
\end{theorem}
The proof, which consists of a careful but elementary case-by-case analysis, is omitted here. The dynamics are piecewise linear; the nonlinearity only appears when trajectories cross the boundaries separating the regions defined by the argument of the function $F$.

\subsubsection{A minimimal stochastic process for a two states fly's flight}
Here, we consider the simplest case in dimension two. There are only two opposite directions, and as described above, we add to the ODE a Brownian perturbation together with a random switching mechanism, i.e., $(V_1(t),V_2(t))$ alternates randomly between the two states $(c,0)$ and $(0,c)$, with $c>0$.
Formally, let $Z_t=(X_t,I_t)\in\mathbb R^2\times\{1,2\}$, where
$X_t=(r_{1t},r_{2t})$ solves
\begin{equation}
\label{eq:EDS2d}
 \left\{\begin{array}{rcl}
   dr_{1t} &=& \big(-r_1 + F(\delta(r_1-r_2)+V_1(I_t))\big)dt
+ \sigma_1 dB_{1t}, \\
dr_{2t} &=& \big(-r_2 + F(\delta(r_2-r_1)+V_2(I_t))\big)dt
+ \sigma_2 dB_{2t},
 \end{array}
 \right.
\end{equation}

with $F(x)=\max(x,0)$, $B_t=(B_t^1,B_t^2)$ be a two–dimensional Brownian motion. The switching process $I_t$ is a continuous-time two-state Markov chain with rate $\lambda>0$: when $I_t=1$, we have $V_1=c$ and $V_2=0$, and when $I_t=2$, we have $V_1=0$ and $V_2=c$.

In other words, $I_t\in\{1,2\}$ is a continuous-time Markov chain, independent of $B_t$, with generator
\[
Q=
\begin{pmatrix}
-\lambda & \lambda\\
\lambda & -\lambda
\end{pmatrix},
\qquad \lambda>0.
\]
We assume:
\[
\delta<1,
\qquad
\sigma_1,\sigma_2>0.
\]
For $i=1,2$ define the drift
\[
b_i(x)
=
- x
+
G\!\big(\delta Lx + V^i\big),
\qquad x\in\mathbb R^2,
\]
where $x=(r_1,r_2)$, $L:\mathbb R^2\to\mathbb R^2$ is linear,
$v_i\in\mathbb R^2$,
$\delta>0$,
and $G((y_1,y_2))=(F(y_1),F(y_2))$. 

Equation \eqref{eq:EDS2d} defines a switching diffusion process, a class of systems for which a well-developed theoretical framework is available; see, for instance, \cite{YinZhu2010}. In the present work, however, we emphasize dynamical systems aspects and adopt a semigroup framework.
Let \[P_t f(z)=\mathbb E_z[f(Z_t)].\]
The following theorem establishes the well-posedness of \eqref{eq:EDS2d}, shows that $(P_t)_{t\geq 0}$ defines a semigroup, and provides the existence and explicit characterization of its generator.Let $C_b$ (resp $C_b^2$) denote the space of continuous (resp. twice continuously differentiable) bounded functions on $\mathbb R^2\times\{1,2\}$. 
\begin{theorem}

\begin{enumerate}
\item For every initial condition $(x,i) \in \mathbb{R}^2 \times \{1,2\}$, equation \eqref{eq:EDS2d} admits a pathwise unique continuous solution.
\item The family $(P_t)_{t \ge 0}$ defines a strongly continuous semigroup on $C_b$. Moreover, for every $f \in C_b^2$,
\[
\frac{d}{dt} P_t f = \mathcal{L} P_t f = P_t \mathcal{L} f,
\]
where $\mathcal{L}$ denotes the generator of $(P_t)$.

\item  The generator $\mathcal{L}$ is given, for $f \in C_b^2$, by
\[
\mathcal L f(x,i)
=
\langle b_i(x),\nabla_x f(x,i)\rangle
+\frac12 \sum_{k=1}^2 \sigma_k^2 \partial_{kk} f(x,i)
+\lambda\big(f(x,j)-f(x,i)\big),
\quad j\neq i.
\]

\end{enumerate}
\end{theorem}
\begin{proof}
We outline the main steps of the proof, deliberately omitting detailed arguments that follow from standard results.
We first remark that for $a,b \in \mathbb R^2$:
\[
|G(a)-G(b)| \le |a-b|.
\]
Since $L$ is linear, there exists $C_L>0$ such that
\[
|Lx-Ly| \le C_L |x-y|.
\]

Thus
\[
|G(\delta Lx+V^i)-G(\delta Ly+V^i)|
\le
|\delta|\, C_L |x-y|.
\]

Therefore
\[
|b_i(x)-b_i(y)|
\le
|x-y| + |\delta|C_L |x-y|
=
C |x-y|,
\]
for some constant $C$ independent of $x,y$.

Hence $b_i$ is globally Lipschitz and of linear growth.
Between successive switching times of $I_t$, the system reduces to a classical SDE of the form
\[
dX_t = b_i(X_t)\,dt + \Sigma\, dB_t,
\]
where $\Sigma=\mathrm{diag}(\sigma_1,\sigma_2)$.
Since $b_i$ is globally Lipschitz and $\Sigma$ is constant,
standard results from SDE theory—based, for instance, on a Picard iteration argument (see, e.g., \cite{oksendal2003sde,karatzas1991brownian})—ensure the existence and uniqueness of a continuous solution.
Moreover, the Markov chain $I_t$ has almost surely finitely many jumps on any compact time interval. One can therefore construct the solution pathwise by solving the SDE successively between jump times. Uniqueness on each interval then yields global uniqueness.
Next, we prove that, $P_t$ maps $C_b$ into itself.
First, for boundedness, we remark that
\[|P_t f|=|E[f(Z_t)]|=|\int f(z)p_{z_t}(z)dz|\leq \|f\|_\infty.\]
where $p_{z_t}$ denotes the density probability of $Z(t)$.
Next, let $X_t^{x}$ and $X_t^{x_n}$ denote the solutions of the SDE starting from
initial conditions $x$ and $x_n$, respectively. For simplicity, we omit the index $i$
The integral formulation of the solutions gives
\[
X_t^{x} = x + \int_0^t b(X_s^{x},I_s)\,ds + \Sigma B_t,
\]
\[
X_t^{x_n} = x_n + \int_0^t b(X_s^{x_n},I_s)\,ds + \Sigma B_t.
\]

Subtracting the two equations,
we obtain
\[
 X_t^{x_n} - X_t^{x} = (x_n-x) + \int_0^t
\big(b(X_s^{x_n},I_s)-b(X_s^{x},I_s)\big)\,ds .
\]

Since the drift is globally Lipschitz, there exists a constant $L>0$ such that
\[
|b(x,i)-b(y,i)| \le L |x-y|
\]
for all $x,y$ and $i\in\{1,2\}$. Therefore
\[
|X_t^{x_n} - X_t^{x}|
\le |x_n-x| + \int_0^t L |X_t^{x_n} - X_t^{x}|\,ds .
\]

Applying Gronwall's inequality yields
\[
|X_t^{x_n} - X_t^{x}| \le |x_n-x| e^{Lt}.
\]

Taking expectations gives
\[
\mathbb E |X_t^{x_n}-X_t^{x}|
\le |x_n-x| e^{Lt}.
\]

Hence, if $x_n\to x$, the right-hand side converges to $0$, which implies
\[
\mathbb E |X_t^{x_n}-X_t^{x}| \to 0.
\]

In particular,
\[
X_t^{x_n} \to X_t^{x}
\]
in probability.
Hence, by standard argument, for $f\in C_b$,
$P_t f$ is continuous.
Thus $P_t$ maps $C_b$ into itself.

Finally, we prove the characterization of the infinitesimal generator.
Let $f\in C_b^2(\mathbb R^2\times\{1,2\})$ and consider the process
$Z_t=(X_t,I_t)$ with $X_t=(r_{1t},r_{2t})$.  
Fix $(x,i)$ and write $b(x,i)=(b_1(x,i),b_2(x,i))$ for the drift:
\[
b_1(x,i)=-x_1+F(\delta(x_1-x_2)+V_1(i)), \qquad
b_2(x,i)=-x_2+F(\delta(x_2-x_1)+V_2(i)).
\]

We compute the infinitesimal generator $A$ defined by
\[
\mathcal L f(x,i) = \lim_{t\to 0} \frac{P_t f(x,i) - f(x,i)}{t},
\]
whenever the limit exists.

Between two switching times of the Markov chain $I_t$, the index $i$ is
constant and the process $X_t$ satisfies the diffusion system
\[
dX_t = b(X_t,i)\,dt + \Sigma\, dB_t,
\]
where $\Sigma=\mathrm{diag}(\sigma_1,\sigma_2)$.

Applying Itô's formula to $f(X_t,i)$ gives
\[
df(X_t,i)
=
\langle b(X_t,i),\nabla_x f(X_t,i)\rangle\,dt
+
\frac12 \sum_{k=1}^2 \sigma_k^2
\partial_{kk} f(X_t,i)\,dt
+
\sum_{k=1}^2 \sigma_k
\partial_k f(X_t,i)\, dB_t^k .
\]

Taking expectations, the stochastic integral has mean zero, so the
expected contribution of the continuous part over a small time interval
$dt$ is
\[
\langle b(x,i),\nabla_x f(x,i)\rangle dt
+
\frac12 \sum_{k=1}^2 \sigma_k^2 \partial_{kk}f(x,i) dt.
\]

\medskip
The switching process $I_t$ is a two-state Markov chain with rate
$\lambda$. Starting from state $i$, during a small time interval $dt$:

\begin{itemize}
\item[-] with probability $1-\lambda dt + o(dt)$, no jump occurs and $I_t$
remains equal to $i$,
\item[-] with probability $\lambda dt + o(dt)$, a jump occurs and the state
changes to $j\neq i$.
\end{itemize}
If a jump occurs, the value of $f$ changes from $f(x,i)$ to $f(x,j)$.
Therefore the expected increment of $f$ due to switching is
\[
\lambda \big(f(x,j)-f(x,i)\big) dt + o(dt).
\]

\medskip

Combining the diffusion and switching contributions and dividing by $dt$,
then letting $dt\to0$, yields the infinitesimal generator
\[
\mathcal L f(x,i)
=
\langle b(x,i),\nabla_x f(x,i)\rangle
+
\frac12 \sum_{k=1}^2 \sigma_k^2 \partial_{kk}f(x,i)
+
\lambda\big(f(x,j)-f(x,i)\big),
\qquad j\neq i .
\]
The other details of the proof are omitted here. Another approach is to verify the Hille-Yoshida theorem.
\end{proof}

Next, we establish the existence and uniqueness of an invariant measure for the process $Z_t$.

\begin{theorem}
\label{th:ErgodicMeasure}
    There exists a unique probability measure $\pi$ such that
\[
\int P_tf \, d\pi = \int f \, d\pi
\quad \text{for all } f \in C_b.
\] Furthemore,$\pi$  satisfies
   \[
\frac1T\int_0^T f(Z_t)dt
\to
\int f\,d\pi
\quad \text{a.s.}
\]
\end{theorem}
To prove theorem \ref{th:ErgodicMeasure}, we first prove the existence of the measure $\pi$ in proposition \ref{prop:exsitence}, uniqueness follows from classical arguments and is omitted. We start with preliminary results.

Define $\mathcal{V}(x,i)=|x|^2=r_1^2+r_2^2$.

\begin{lemma}
There exist constants $\alpha>0$ and $\beta>0$ such that
\[
\mathcal L \mathcal{V}(x,i) \le -\alpha \mathcal{V}(x,i)+\beta.
\]
Furthermore, one may take any $\alpha<2(1-\delta)$.
\end{lemma}

\begin{proof}
We compute
\[
\mathcal L \mathcal{V}((r_1,r_2,i)
=
-2(r_1^2+r_2^2)+r_1F(\delta(r_1-r_2)+V_1(i))+r_2F(\delta(r_2-r_1)+V_2(i))
+\sigma_1^2+\sigma_2^2.
\]
This implies,
\begin{equation}
\label{eq:LV}
    \mathcal L \mathcal{V}
\leq 
-2(1-\delta-\nu)(r_1^2+r_2^2)+c\frac{c^2}{8\nu}
+\sigma_1^2+\sigma_2^2.
\end{equation}

where $\nu $ can be chosen as small as desired,
which yields the result.
\end{proof}

\begin{lemma}
    The following estimate holds,
    \[
\mathbb E \mathcal{V}(Z_t)
\le
e^{-\alpha t}\mathcal{V}(Z_0)
+\frac{\beta}{\alpha},
\]
hence 
\[
\sup_{t\ge0}\mathbb E|X_t|^2<\infty.
\]
\end{lemma}
\begin{proof}
    We first apply It\^o's formula to $dV$,
    \[
\begin{aligned}
d\mathcal{V}(Z_t) 
&= \frac{\partial V}{\partial r_1} dr_1 + \frac{\partial V}{\partial r_2} dr_2 
   + \frac{1}{2} \frac{\partial^2 V}{\partial r_1^2} (\sigma_1^2) 
   + \frac{1}{2} \frac{\partial^2 V}{\partial r_2^2} (\sigma_2^2) \\
&= 2 r_1\, dr_1 + 2 r_2\, dr_2 + (\sigma_1^2 + \sigma_2^2)\, dt \\
&= 2 r_1 b_1(r_1,r_2,i)\, dt + 2 r_2 b_2(r_1,r_2,i)\, dt + (\sigma_1^2 + \sigma_2^2)\, dt
   + 2 r_1 \sigma_1\, dW_1 + 2 r_2 \sigma_2\, dW_2 \\
&= 2 r_1 b_1 + 2 r_2 b_2 + \sigma_1^2 + \sigma_2^2 
   + 2 r_1 \sigma_1\, dB_1 + 2 r_2 \sigma_2\, dB_2.
\end{aligned}
\]
Integrating, between $0$ and $t$, and taking expectations, we obtain, 
\[
\mathbb{E}[\mathcal{V}(Z_t)] = \mathcal{V}(Z_0) + \int_0^t \mathbb{E}[\mathcal{L}\mathcal{V}(Z_s)] \, ds.
\]
Using \eqref{eq:LV} gives
\[\mathbb{E}[\mathcal{V}(Z_t)] \leq \mathcal{\mathcal{V}}(Z_0)-\alpha \int_0^t \mathbb{E}[\mathcal{L}V(Z_s)]+\beta t\]
 Set 
 \[u(t)=\mathbb{E}[V(Z_t)],\]
 then we have
 \[u(t)\leq u(0)-\alpha \int_0^tu(s)ds+\beta t\]
 which, thanks to a Gronwall estimate type gives
\[u(t)\leq e^{-\alpha t}u(0)+\frac{\beta}{\alpha}(1-e^{-\alpha t}).\]
\end{proof}

We now  define empirical measures
\[
\mu_T(A)
=
\frac1T\int_0^T \mathbb P(Z_t\in A)\,dt.
\]
and prove the following convergence result.
\begin{proposition}[Existence]
\label{prop:exsitence}
There exists a subsequence $\mu_{T_k}$ and a measure $\pi$ such that
\[
\int f \, d\mu_{T_k} \;\longrightarrow\; \int f \, d\pi 
\quad \text{for all } f \in C_b.
\]
Furthermore, $\pi$ satisfies:
\begin{equation}
\label{eq:invariance}
\int P_tf \, d\pi = \int f \, d\pi
\quad \text{for all } f \in C_b.
\end{equation}
\end{proposition}
\begin{proof}
    We first prove that for every $\varepsilon > 0$ there exists a compact set $K \subset X$ such that
\begin{equation}
    \label{eq:tightness}
    \mu_T(K) \ge 1 - \varepsilon \quad \text{for all } T >0 .
\end{equation}
For $R>0$,
\[P(|Z_t|\leq R)=1-P(|Z_t|> R)\]
Next we remark that, thanks to the Bienayme-Chebyshev inequality,
\[
\mathbb P(|X_t|>R)
\le
\frac{\mathbb E|X_t|^2}{R^2}
\le
\frac{C}{R^2}.
\]
where the constant $C$ does not depend on $t$  nor in $i$. We deduce that
\[
P(|Z_t|>R)
\le
\frac{C}{R^2}.
\]
This proves \eqref{eq:tightness}. Next, a classical argument (Prokhorov's theorem) proves that  there exists a subsequence $\mu_{T_k}$ and a measure $\pi$ such that
\[
\int f \, d\mu_{T_k} \;\longrightarrow\; \int f \, d\pi 
\quad \text{for all } f \in C_b.
\]

From the semigroup properties, we deduce that $\pi$ satisfies \eqref{eq:invariance}. Indeed,
\[
\int P_t f \, d\mu_T 
= \frac{1}{T} \int P_t f(z)\int_0^T p_{z_s} \, ds\,dz=\frac{1}{T} \int_0^T \int P_t f(z) p_{z_s}\,dz \, ds
\]
where $p_{z_s}$ denotes the density probability of $Z(s)$. It follows that
\[
\int P_t f \, d\mu_T 
= \frac{1}{T} \int_0^T \mathbb{E}[P_t f(Z_s)] \, ds,
\]
\[
= \frac{1}{T} \int_0^T \mathbb{E}[f(Z_{s+t})] ds,
\]
\[
= \frac{1}{T} \int_t^{T+t} \mathbb{E}[f(Z_u)] du.
\]
Thus,
\[
\int P_t f \, d\mu_T-\int f \, d\mu_T
= \frac{1}{T} \int_t^{T+t} \mathbb{E}[f(Z_u)] du-\frac{1}{T} \int_0^{T} \mathbb{E}[f(Z_u)] du .
\]
\[
= \frac{1}{T} \bigg(\int_t^{T+t} \mathbb{E}[f(Z_u)] du-\frac{1}{T} \int_0^{T} \mathbb{E}[f(Z_u)]\bigg) du ,
\]
\[
= \frac{1}{T} \bigg(\int_T^{T+t} \mathbb{E}[f(Z_u)] du-\frac{1}{T} \int_0^{t} \mathbb{E}[f(Z_u)]\bigg) du.
\]
Since $f$ is bounded, we deduce that,
\[
\bigg|\int P_t f \, d\mu_T-\int f \, d\mu_T\bigg|\leq \frac{Ct}{T} \mbox{ for some constant } C. 
\]
Letting $T_k \to \infty$, we obtain,
\[
\int P_t f \, d\pi = \int f \, d\pi.
\]
\end{proof}

\begin{proof}[Proof of theorem \ref{th:ErgodicMeasure}]
To prove the theorem, it remains to prove that $\pi$ is unique. This follows from a standard argument in markov chains theory. We omit the details here and refer to \cite{meyn2009markov},\cite{YinZhu2010}. We also deduce that
  \[
\frac1T\int_0^T f(Z_t)dt
\to
\int f\,d\pi
\quad \text{a.s.}
\]
\end{proof}
\subsection{Qualitative Analysis for $N\in \mathbb{N}, N>2$}
Solutions of Equation \eqref{eq:linReLu} can be expressed, for $t \in (t_k, t_{k+1})$, as
\[
r(t) = e^{B_k (t - t_k)} r(t_k) + \int_{t_k}^{t} e^{B_k (t - s)} V(s)\, ds,
\]
where $B_k$ is the matrix associated with the active set of neurons on the interval $(t_k, t_{k+1})$.

More precisely, the state space is partitioned according to the regions
\[
\mathcal{A}_i = \left\{ r \in \mathbb{R}^N \;:\; \sum_{j=0}^{N-1} A_{ij} r_j + V_i > 0 \right\}, 
\quad i \in \{0,\dots,N-1\}.
\]

The subdivision $t_1,t_2...$ is such that on each interval $(t_k, t_{k+1})$, the trajectory remains in a region where the set $\mathcal{A}(t)$ of active neurons is fixed, that is, 
\[
    \mathcal{A}=\{j\in\{0,...,N-1\}; \sum_{j=0}^{N-1} A_{ij} r_j + V_j > 0\}
\]
do not change. As long as the trajectory remains within this region, the nonlinearity $F$ acts linearly, and the dynamics reduce to a linear system with matrix $B_k$. We are however proving a more qualitative result. We assume that $V_0>0$ and $V_i=0$ for $i\neq 0$. 

\begin{proposition}
\label{prop:m_xformulation}
We can rewrite Equation~\eqref{eq:linReLu}  as
\begin{equation}
    r_i' = -r_i + F\Big( \delta \big( m_x \cos\theta_i + m_y \sin\theta_i \big) + V_i \Big),
\end{equation}
or equivalently,
\begin{equation}
    r_i' = -r_i + F\Big( \delta R \cos(\theta_i - \phi) + V_i \Big),
\end{equation}
with
\[
m_x = \sum_{j=0}^{N-1} r_j \cos\theta_j, 
\qquad
m_y = \sum_{j=1}^{N-1} r_j \sin\theta_j,
\]
and with $(R, \phi)$ defined as the polar coordinates of $m_x + i m_y$ ,  
\[
m_x + i m_y=Re^{i\phi}
\]
Furthermore, if we assume that
\[\delta<\frac{1}{\sum_{i=1}^{N-1}\sin^2\theta_i},\]
then,
 \[\lim_{t\rightarrow +\infty}m_y=0\]
 and, 
 \[\forall j\in \{1,...,N-1\}, \lim_{t\rightarrow +\infty}r_j-r_{N-j}=0. \]
\end{proposition}
\begin{proof}
    We simply apply the formula,
    \[\cos(\theta_i-\theta_j)=\cos\theta_i\cos\theta_j+\sin\theta_i\sin \theta_j,\]
    therefore
    \[ r_i'=-r_i+F( \delta \sum_j \left(\cos\theta_i\cos\theta_j+\sin\theta_i\sin \theta_j\right)r_j +V_i )\]
 \[ =-r_i+F( \delta \left( m_x \cos\theta_i + m_y \sin\theta_i\right) +V_i )\]
  \[ =-r_i+F( \delta R\cos(\theta_i - \phi)  +V_i ).\]
Next, we compute $m_y'$. We have,
 \[m_y'=\sum_{j=1}^{N-1} r'_j \sin\theta_j\]
  \[=\sum_{j=1}^{N-1} \bigg(-r_i+F( \delta \left( m_x \cos\theta_i + m_y \sin\theta_i\right) +V_i )\bigg) \sin\theta_j\]
    \[=-m_y+\sum_{j>0,j\in \mathcal{A}} \delta \left( m_x \cos\theta_j \sin\theta_j + m_y \sin^2\theta_j\right) \]
    where $\mathcal{A}$ denotes the active indexes,
    \[\mathcal{A}=\{j\in\{0,...,N-1\}; \delta \left( m_x \cos\theta_j + m_y \sin\theta_j\right) +V_j>0\}.\]
    Now, since 
    \[\cos\theta_j=\cos\theta_{N-j}, \,\, \sin\theta_j=-\sin\theta_{N-j},\]
    we obtain,
    \[m_y'=-m_y+\delta m_y\sum_{j>0,j\in \mathcal{A}}\sin^2\theta_j,\]
    which proves that if
    \[\delta<\frac{1}{\sum_{i=1}^{N-1}\sin^2\theta_i},\]
    $m_y$ converges exponentially toward $0$.
    Finally, we remark that
    \[r_i'-r'_{N-i}=-(r_i-r_{N-i})+2\delta m_y\sum_{j>0,j\in \mathcal{A}}\sin^2\theta_j\]
    which proves the result since $m_y$ converges to $0$.

\end{proof}
We can now establish the existence of a locally stable bump centered at $\theta_0=0$.
\begin{theorem}
\label{th:steadystate}
We assume that
\[\delta<\min\bigg\{\frac{1}{\sum_{\theta_i \in \{-\frac{\pi}{2},\frac{\pi}{2}\}} \cos^2 \theta_i},\frac{1}{\sum_{i=1}^{N-1}\sin^2\theta_i}\bigg\},\]
then there exists a stationary solution $\bar{r}=\bar{r}_0,..,\bar{r}_{N-1}$ satisfying
\[
\bar{r}_0=\delta m_{\bar{x}}+V_0,\, \,\bar{r}_j=\left\{\begin{array}{c}
    \delta m_{\bar{x}} \cos\theta_j, if\,\, j>0\,\, and \,\,\theta_j \in \{-\frac{\pi}{2},\frac{\pi}{2}\} \\
    0 \,\  otherwise 
\end{array}\right.
\]
with,
\[
m_x=\frac{V_0}{1-\delta\sum_{\theta_i \in \{-\frac{\pi}{2},\frac{\pi}{2}\}} \cos^2 \theta_i} ,\,\, m_y=0.
\]
Furthermore, $\bar{r}$ is locally asymptotically stable in the positive space $r_j>0, j\in\{0,...,N-1\}$
\end{theorem}
\begin{proof}
    Under the assumptions, it follows from the proof of proposition \ref{prop:m_xformulation} that $m_y'=0$ implies $m_y=0$. Next we compute $m_x'$. Adopting the same notations as before, looking for a solution $m_x>0$, we find
    \[m_x'=-m_x+\delta m_x\sum_{j\in \mathcal{A}}\cos^2\theta_j+V_0.\]
    Setting $m_x'=0$, we obtain,
    \[m_x=\frac{V_0}{1-\delta\sum_{i\in \mathcal{A}} \cos^2 \theta_i}.\]
    Choosing $\delta$ small enough ensures $m_x>0$ and $\theta_j \in \{\frac{\pi}{2},\frac{\pi}{2}\}\Leftrightarrow j\in \mathcal{A}$. The local stablity follows from the positive invariance of the spaces $\mathcal{A}_j, j\in \mathcal{A}$ and  $\mathcal{A}^c_j, j\notin \mathcal{A}$ if $m_y$ close enough to zero.
     
\end{proof}
The next section is devoted to numerical simulations.
\section{Numerical Simulations}
In this section, we present numerical . We begin with the case 
$N=2$, and then consider a higher-dimensional setting with $N=50$. The simulations were implemented in both $C^{++}$ and Python; the figures shown here are generated using the Python implementation.
\subsection{Numerical simulations for $N=2$}
In this subsection, we consider equation \eqref{eq:EDS2d}. For the numerical simulations, we take 
\[\delta = 0.2,\,\sigma_1 = \sigma_2=0.1,\, \lambda = 0.1, \,c = 1.
\] 
These values were selected to distinguish between small continuous fluctuations in the fly’s motion and larger changes in direction. In particular, $\delta=0.2<\tfrac12$ places the deterministic dynamics in the weak-coupling regime, where each fixed cue position is associated with a unique globally asymptotically stable winner equilibrium. With $c=1$, the corresponding equilibria are $E_1=(1.25,0)$ and $E_2=(0,1.25)$. The relatively small noise amplitudes $\sigma_1=\sigma_2=0.1$ generate moderate Brownian fluctuations in the trajectories and around the equilibrium selected by the current cue, rather than overwhelming the cue-driven response. Meanwhile, $\lambda=0.1$ gives a mean residence time $1/\lambda=10$ in each cue state, which is substantially longer than the deterministic relaxation time. Hence, after each cue switch, the system generally has sufficient time to approach the vicinity of the corresponding stable equilibrium before the next switch occurs. Biologically, the Brownian terms represent small ongoing fluctuations during flight, whereas switches of $I_t$ represent less frequent and more pronounced changes in the directional cue.
These parameter values are not intended as direct experimental estimates; rather, they are chosen to illustrate the qualitative stochastic dynamics in a regime where cue tracking, local fluctuations, and directional switching occur on clearly separated scales.
The numerical results are summarized in Figure \ref{fig:N2}. The top-left panel shows a realization of the Markov switching process $I(t)$. The top-right panel displays the trajectories of $r_1(t)$ and $r_2(t)$. The bottom-left panel illustrates the corresponding phase-space trajectory. Finally, the bottom-right panel shows the empirical asymptotic distribution of $r_1$. It illustrates that, when $I_t=1$, the activity $r_1$ increases toward the corresponding winner equilibrium while $r_2$ decreases; when $I_t=2$, their roles are reversed. We can also observe the small, continuous fluctuations around this cue-driven behavior without obscuring the alternating dominance of the two populations. The empirical asymptotic distribution of $r_1$ exhibits a bimodal structure, with two peaks concentrated near $0$ and $c$ respectively. Overall, these dynamics capture well the behavior of the system, driven by external cues and intrinsic noise, and are consistent with the expected biological response of the fly. Simulations with other noise intensities were also performed, but they did not reproduce this intended balance between local stochastic fluctuations and reliable cue tracking; these additional regimes will be presented for comparison in the associated PhD thesis of the first author. 
\begin{figure}[t]
    \centering
     \includegraphics[width=5cm,height=3cm]{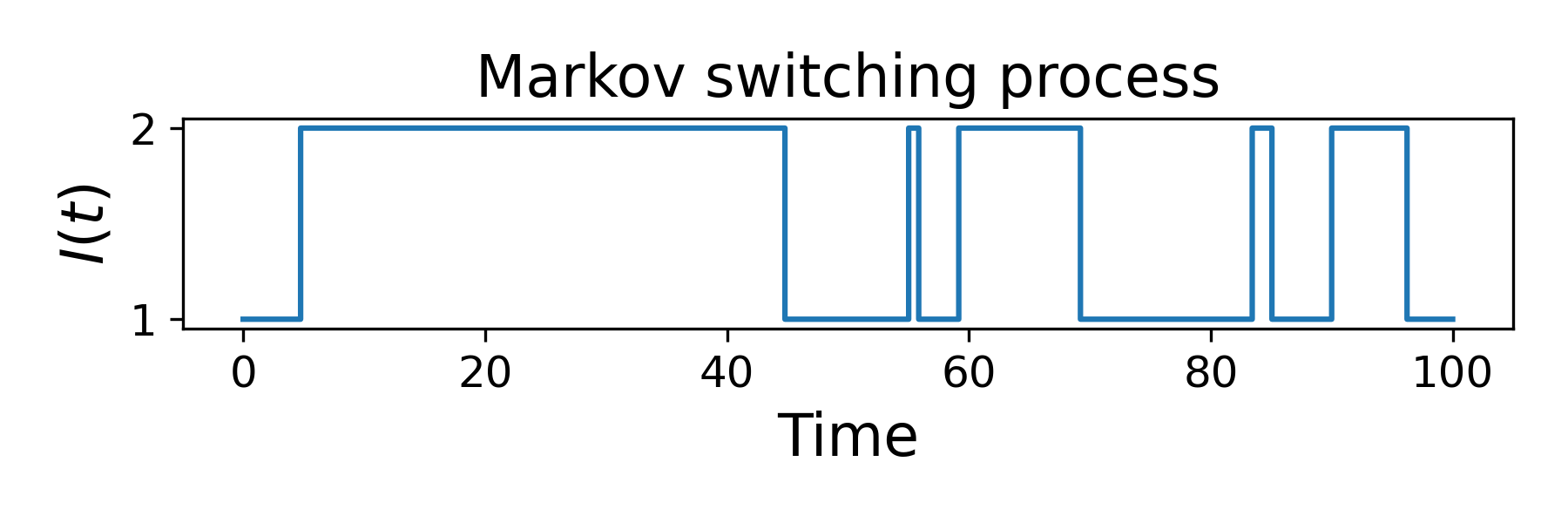}
        \includegraphics[width=5cm]{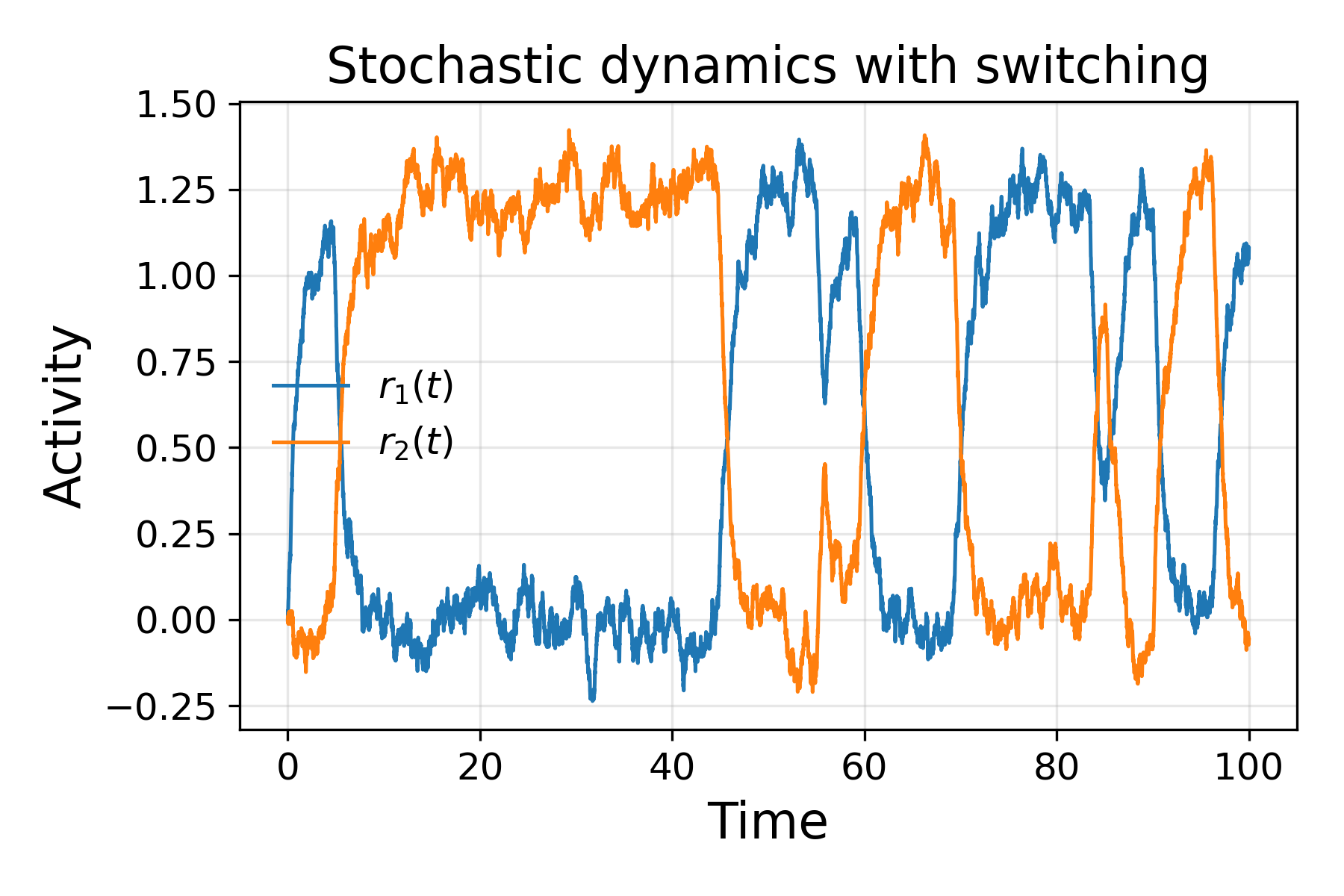}\\
        \includegraphics[width=5cm,height=4cm]{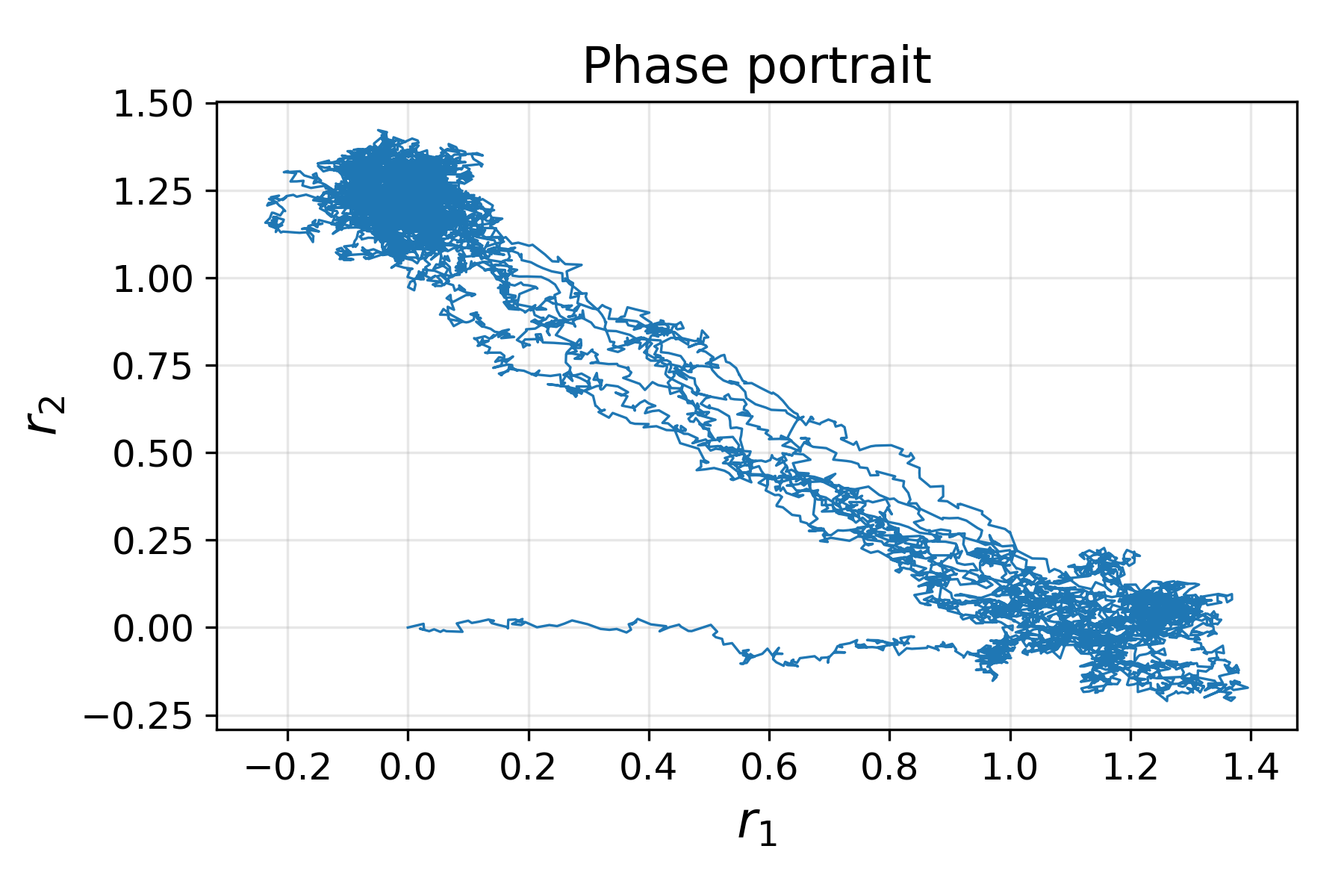}
        \includegraphics[width=5cm,height=4cm]{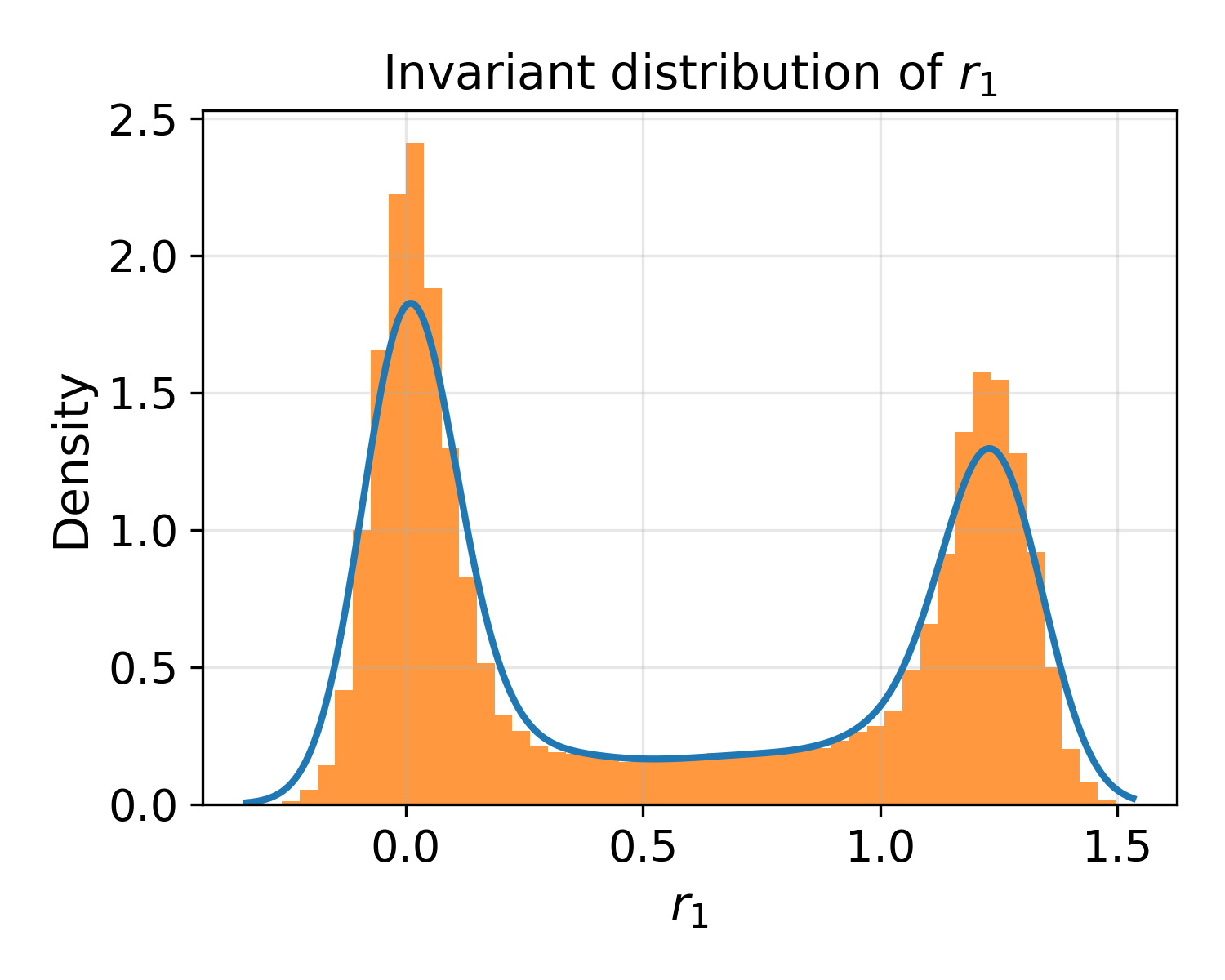}
 \caption{Simulations of equation \eqref{eq:EDS2d} with parameters $\delta = 0.2$, $\sigma_1 = \sigma_2 = 0.1$, $\lambda = 0.1$, and $c = 1$. The top-left panel shows a realization of the Markov switching process $I(t)$. The top-right panel displays the trajectories of $r_1(t)$ and $r_2(t)$. The bottom-left panel shows the corresponding phase-space trajectory. The bottom-right panel illustrates the empirical asymptotic distribution of $r_1$. We observe that, when $I_t=1$, the activity $r_1$ increases toward the corresponding winner equilibrium while $r_2$ decreases; when $I_t=2$, their roles are reversed. We can also observe the small, continuous fluctuations produced by the Brownian motion around this cue-driven behavior without obscuring the alternating dominance of the two populations.The dynamics exhibit a symmetric switching behavior: neural activity is high for the neuron aligned with the current cue and low for the other, with roles reversing after each switch. The empirical distribution of $r_1$ is bimodal, with peaks near $0$ and $c$. In this realization, a slightly higher mass is observed near $0$ for $r_1$ (and correspondingly near $c$ for $r_2$).}
    \label{fig:N2}
\end{figure}

\subsection{Numerical simulations for $N=50$}
In this subsection, we first consider the deterministic network \eqref{eq:linReLu} and then the stochastic extension \eqref{eq:sdeN} for $N=50$.  This choice of $N$ is biologically motivated by experimental counts of approximately 53–68 E-PG neurons in Drosophila \cite{giraldo2018}. Each neuron is assigned a preferred angle $\theta_i=2\pi i/N$, so the network provides a discrete neuronal representation of the angular domain. Figure \ref{fig:Nd_sim} summarizes the main numerical observations under different initial conditions in the deterministic case. In Panel (A), we show representative time courses for selected indices $i \in \{0,1,49\}$, starting from a strongly asymmetric initial condition. Despite this imbalance, the system converges toward the stationary profile $\bar{r}$ identified in Theorem \ref{th:steadystate}, with the expected spatial structure $\bar{r}_i = m_{\bar{x}} \cos\left(\frac{2\pi i}{50}\right)$. Panel (B) depicts the asymptotic profile obtained from random initial conditions in $(0,1)$, illustrating convergence toward the same stationary state and its characteristic spatial structure. Finally, Panel (C) presents the full spatio-temporal evolution via a heat map, clearly showing the relaxation toward equilibrium.
Finally, figure \ref{fig:Nsto} illustrates the simulation of \eqref{eq:sdeN} with the following parameters:
\[\delta = 0.05,\,\sigma_i=0.1,\, \lambda = 0.1, \,c = 1.
\] 
The left panel shows the angle $\theta_i$ corresponding to $I_t=i+1\in\{1,...,N$\} while the right panel presents the full spatio-temporal evolution as a time dependent heat map.
Overall, the simulations accurately capture the system’s behavior under the combined effect of external cues and intrinsic noise, and are consistent with the expected biological response of the fly across all directions.

\begin{figure}[t]
    \centering
    \begin{subfigure}[t]{0.3\textwidth}
        \centering
        \includegraphics[width=4cm,height=3cm]{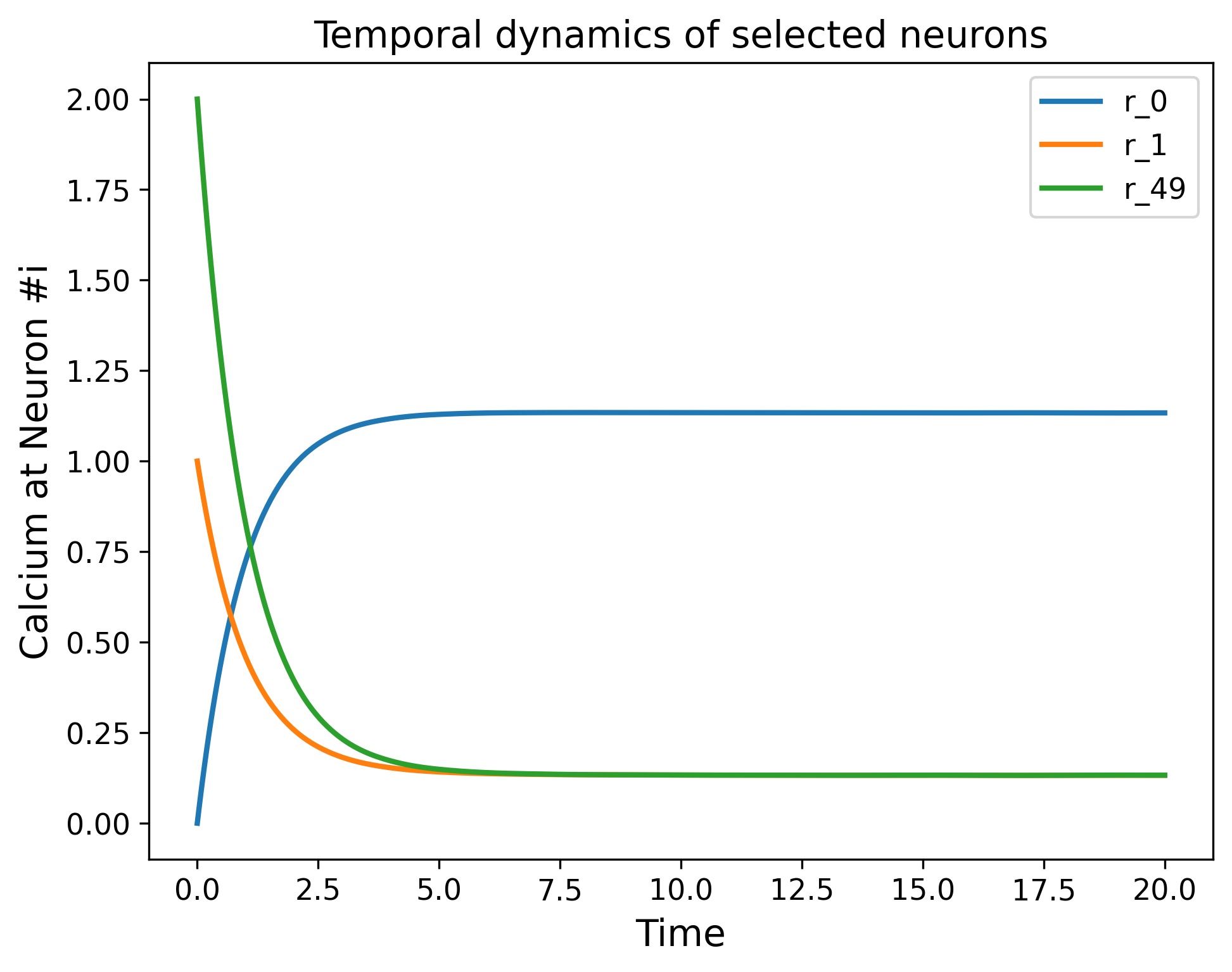}
        \caption{Time course for $i\in\{0,1,49\}$. The following  dyssimetric initial condition was considered  $r_1(0)=1, r_{49}(0)=2$, $r_j(0)=0$ for $j\neq 0$. Asymptotically the trajectory converges toward the stationary solution $\bar{r}$ from theorem \ref{th:steadystate}, with in particular   $\bar{r}_1=\bar{r}_{49}=m_{\bar{x}}\cos\frac{2\pi}{50}$. }
        \label{fig:sub1}
    \end{subfigure}
    \hfill
    \begin{subfigure}[t]{0.3\textwidth}
        \centering
        \includegraphics[width=4cm,height=3cm]{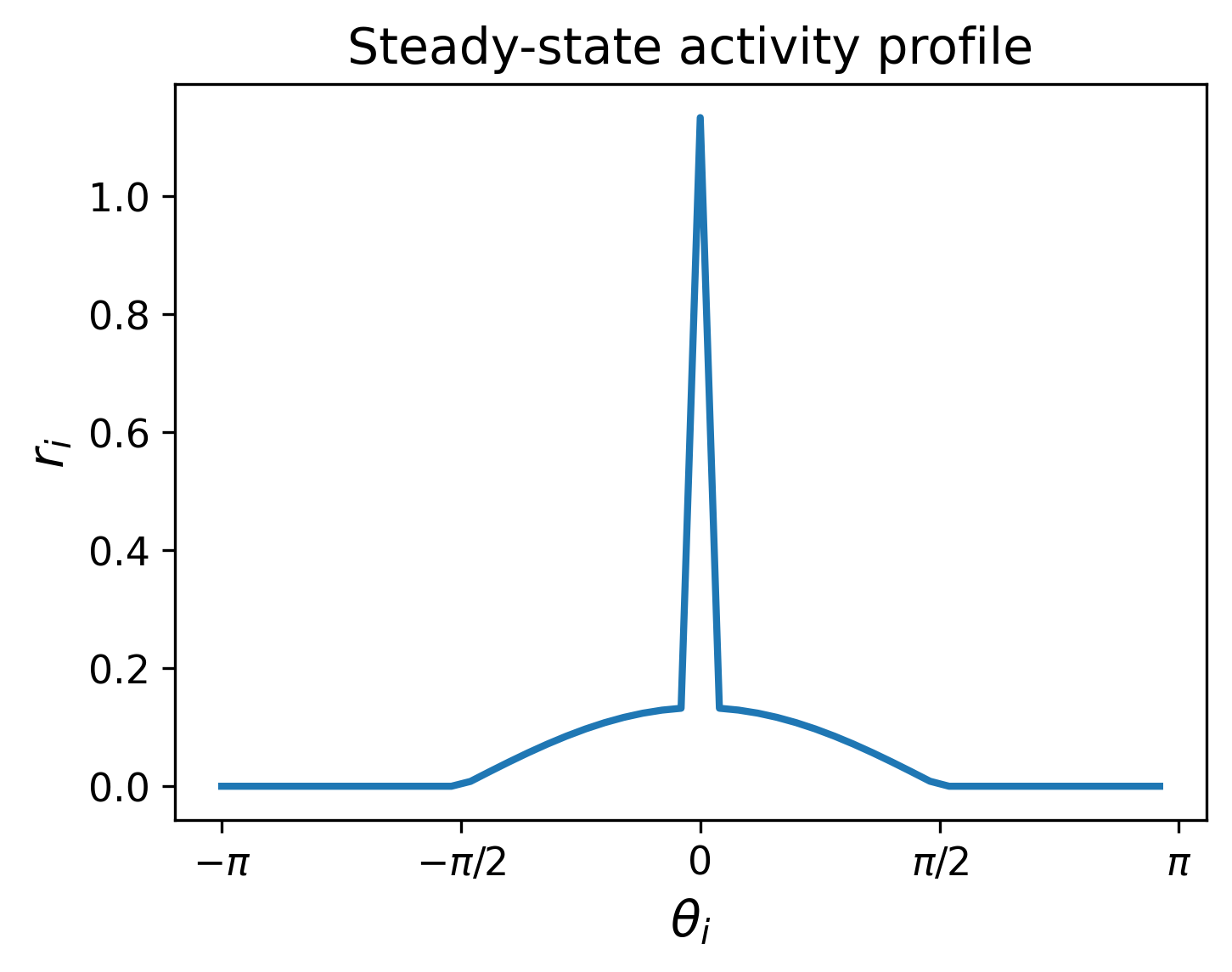}
        \caption{Asymptotic stationary solution with a bump at $0$. Initial conditions were chosen randomly in $(0,1)$. We observe that the solution converges toward the stationary distribution identified in Theorem \ref{th:steadystate}. }
        \label{fig:sub2}
    \end{subfigure}
    \hfill
    \begin{subfigure}[t]{0.3\textwidth}
        \centering
        \includegraphics[width=4cm,height=3cm]{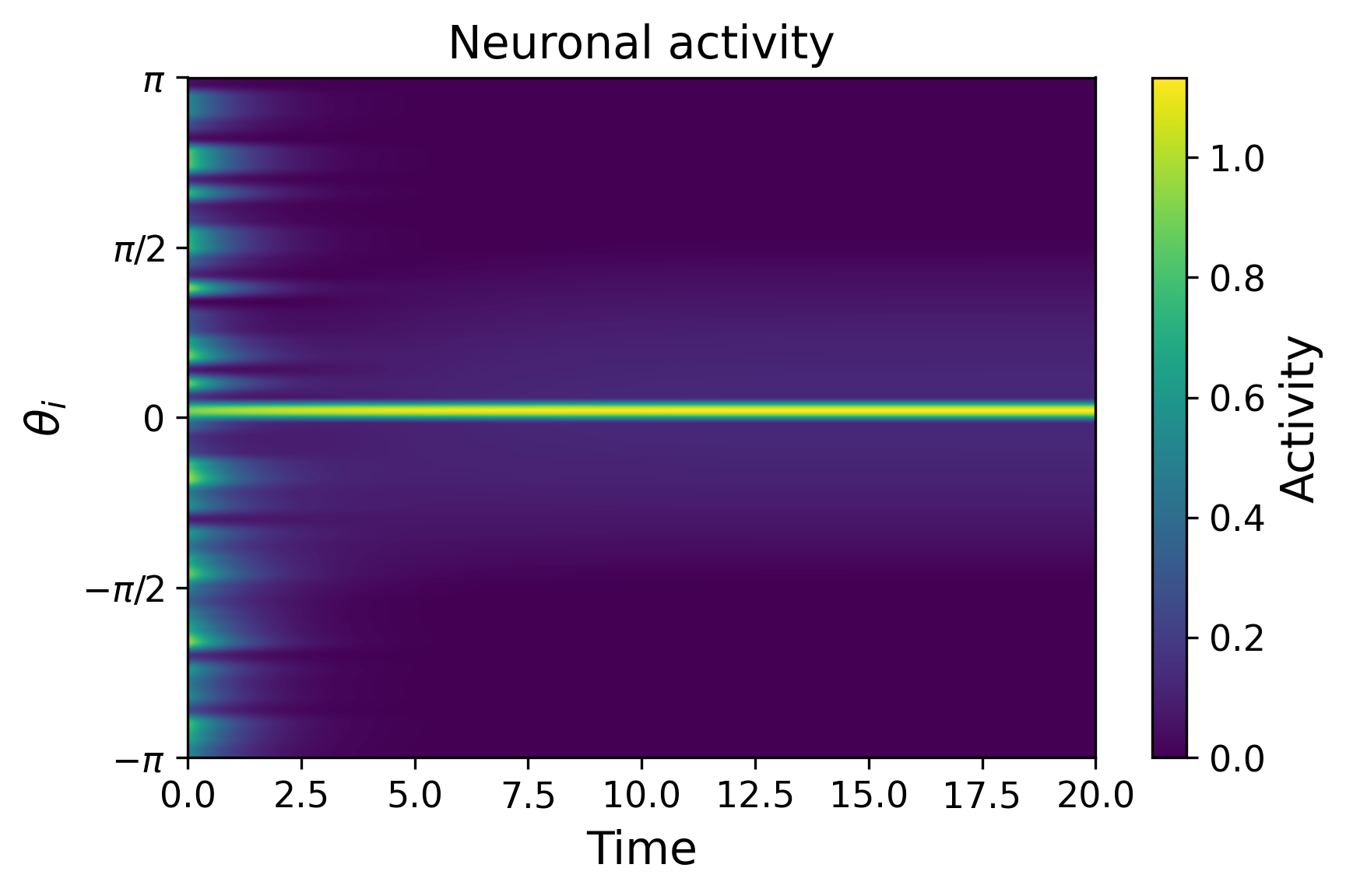}
        \caption{Same solution as in Panel B, but including time and transient trajectories represented as a heat map. Starting from random initial conditions in $(0,1)$, we observe convergence toward the stationary distribution identified in theorem \ref{th:steadystate}.}
        \label{fig:sub3}
    \end{subfigure}

    \caption{
        Simulation of equation \eqref{eq:linReLu} with $N=50$, $\delta=0.05$ and different initial conditions. We observe that the stationary solution described in theorem \ref{th:steadystate} attracts these initial conditions.
    }
    \label{fig:Nd_sim}
\end{figure}
\begin{figure}[t]
    \centering
     \includegraphics[width=12cm]{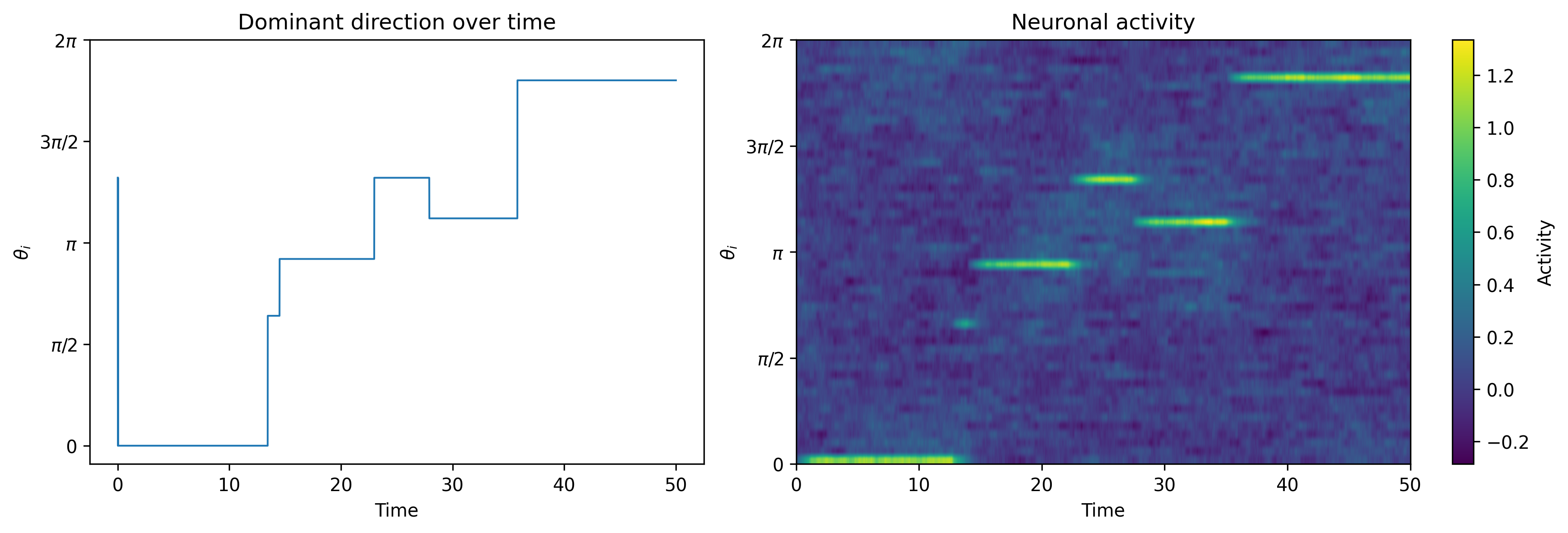}

       \caption{
Simulation of the network \eqref{eq:sdeN} for $N=50$ and with $\delta = 0.05,\,\sigma_i=0.1,\, \lambda = 0.1, \,c = 1$. The left panel depicts the angle $\theta_i$ associated with $I_t = i+1 \in \{1,\ldots,N\}$, while the right panel displays the full spatio-temporal evolution in the form of a time-dependent heat map. Overall, the simulations faithfully reproduce the system’s behavior under the joint influence of external cues and intrinsic noise, and are consistent with the expected biological response of the fly across all directions.
}
    \label{fig:Nsto}
\end{figure}

\section{Conclusion}

In this article, we have introduced and analyzed stochastic and deterministic neural field models inspired by the ellipsoid body's architecture of the Drosophila central complex. Starting from a biologically motivated connectivity structure, we derived a reduced tractable model capable of sustaining localized activity bumps and encoding angular information.
We provided a mathematical analysis of both the deterministic and stochastic dynamics. In particular, we established well-posedness, characterized the generator of the associated switching diffusion, and proved the existence of a stationary regime. In the deterministic setting, we identified parameter regimes ensuring global convergence toward a stable bump solution. In the stochastic case, we showed that the interplay between noise and external cues preserves the qualitative structure of the dynamics while inducing realistic switching and variability in neural activity. From a biological perspective, these results contribute to the ongoing effort to understand how low-dimensional population codes emerge from recurrent neural circuits in the insect brain. The Drosophila central complex is currently an active area of experimental and theoretical research, and increasingly detailed connectomic data continue to refine our understanding of its functional architecture. Our work provides a mathematically grounded framework that can be extended to incorporate more realistic connectivity patterns and sensory inputs, and may serve as a basis for future studies linking circuit-level mechanisms to navigation and decision-making in biological systems.

\medskip
Received xxxx 20xx; revised xxxx 20xx; early access xxxx 20xx.
\medskip

\end{document}